\documentclass[draft]{amsart}
\usepackage{amsmath}
\usepackage{amssymb}
\newtheorem{theorem}{Theorem}[section]
\newtheorem{lemma}[theorem]{Lemma}

\newtheorem{corollary}[theorem]{Corollary}
\theoremstyle{definition}

\theoremstyle{remark}
\newtheorem{remark}[theorem]{Remark}

\numberwithin{equation}{section}

\begin{document}

\title[Quantization errors for in-homogeneous self-similar measures]{Asymptotic quantization errors for in-homogeneous self-similar measures supported on self-similar sets}
\author{Sanguo Zhu}
\address{School of Mathematics and Physics, Jiangsu University of Technology,
Changzhou 213001, China.} \email{sgzhu@jsut.edu.cn}
\thanks{The author is supported by China Scholarship Council (No. 201308320049).}

\begin{abstract}
We study the quantization for a class of in-homogeneous self-similar measures $\mu$ supported on self-similar sets. Assuming the open set condition for the corresponding iterated function system, we prove the existence of the quantization dimension for $\mu$ of order $r\in(0,\infty)$ and determine its exact value $\xi_r$. Furthermore, we show that, the $\xi_r$-dimensional lower quantization coefficient for $\mu$ is always positive and the upper one can be infinite. We also give a sufficient condition to ensure the finiteness of the upper quantization coefficient.
\end{abstract}

\maketitle

\section{Introduction}

With a deep background in information theory, the quantization problem for probability measures has been studied intensively in the past decades. One of the main aims for mathematicians is to study the asymptotic errors in the approximation of a given probability measure with discrete probability measures of finite support. For rigorous mathematical foundations of this theory, we refer to Graf and Luschgy \cite{GL:00}. Further theoretical results and promising applications are contained in \cite{BW:82,GL:01,GL:04,GL:12,GN:98,Gru:04,PG:98,PK:01,Za:82}. Next, we recall some important objects in quantization theory.

We write $\mathcal{D}_n:=\{\alpha\subset\mathbb{R}^q:1\leq{\rm card}(\alpha)\leq n\}$ for every $n\in\mathbb{N}$.
Let $P$ be a Borel probability measure on $\mathbb{R}^q$ and let $0\leq r<\infty$. The $n$th
quantization error for $P$ of order $r$ is given by \cite{GL:00,GL:04}
\begin{eqnarray}
e_{n,r}(P):=\left\{ \begin{array}{ll}
\inf_{\alpha\in\mathcal{D}_{n}}\big(\int d(x,\alpha)^{r}dP(x)\big)^{1/r},\;\;\;\;\;\; r>0,\\
\inf_{\alpha\in\mathcal{D}_{n}}\exp\int\log d(x,\alpha)dP(x),\;\;\;\;\; r=0.
\end{array}\right.\label{quanerror}
\end{eqnarray}
If the infimum in (\ref{quanerror}) is attained at some
$\alpha\subset R^{d}$ with $1\leq{\rm card}(\alpha)\leq n$,  we call
$\alpha$ an $n$-optimal set of $P$ of order $r$.  The collection
of all the $n$-optimal sets of order $r$ is denoted by
$C_{n,r}(P)$.

According to \cite{GL:00}, $e_{n,r}(P)$ equals the error in the approximation of $P$ with discrete measures supported on at most $n$ points, in the sense of $L_r$ metrics. With some natural restrictions, $e_{n,r}(P)$ tends to $e_{n,0}(\mu)$ as $r\to 0$ \cite{GL:04}. We also call $e_{n,0}(P)$ the $n$th geometric mean error for $P$. The upper and lower quantization dimension for $P$ of order $r$, which are defined below, characterize the asymptotic quantization error in a natural manner:
\begin{equation*}
\overline{D}_{r}(P):=\limsup_{n\to\infty}\frac{\log n}{-\log
e_{n,r}(P)};\;\;\underline{D}_{r}(P):=\liminf_{n\to\infty}\frac{\log
n}{-\log e_{n,r}(P)}.
\end{equation*}
If $\overline{D}_{r}(P)=\underline{D}_{r}(P)$, we call the common value the
quantization dimension of $P$ of order $r$ and denote it by $D_r(P)$. In case that $D_r(P)=:s$ exists, we are further concerned with the $s$-dimensional upper and lower quantization coefficient which are defined below (cf. \cite{GL:00,PK:01}):
\begin{eqnarray*}
\overline{Q}_r^s(P):=\limsup_{n\to\infty}n^{\frac{1}{s}}e_{n,r}(P),\;\;
\underline{Q}_r^s(P):=\liminf_{n\to\infty}n^{\frac{1}{s}}e_{n,r}(P).
\end{eqnarray*}
These two quantities provide us with more accurate information for the asymptotics of the quantization error than the quantization dimension. So far, the upper and lower quantization coefficients have been well studied for absolutely continuous measures \cite[Theorem 6.2]{GL:00}  and some classes of fractal measures, including self-similar measures \cite{GL:00} and dyadic homogeneous Cantor measures \cite{Kr:08}. In the present paper, we will further study the quantization problem for in-homogeneous self-similar measures. For this purpose, we need to recall some related definitions.

Let $(f_i)_{i=1}^N$ be a family of contractive similitudes on
$\mathbb{R}^q$ with contraction ratios $(s_i)_{i=1}^N$.
According to \cite{Hu:81}, there exists a
unique non-empty compact subset $E$ of $\mathbb{R}^q$ such that
$E=f_1(E)\cup f_2(E)\cup\cdots\cup f_N(E)$.
The set $E$ is called the self-similar set associated with $(f_i)_{i=1}^N$.

We say that $(f_i)_{i=1}^N$ satisfies the strong separation condition (SSC) if the sets
$f_i(E),i=1,\cdots ,N$, are pairwise disjoint; we say that it
satisfies the open set condition (OSC) if there exists a non-empty open
set $U$ such that $f_i(U)\cap f_j(U)=\emptyset$ for all $i\neq j$
and $f_i(U)\subset U$ for all $i=1,\cdots,N$; if this open set $U$ can be chosen such that $U\cap E\neq\emptyset$, then we say $(f_i)_{i=1}^N$ satisfies the strong open set condition (SOSC). By \cite{Schief:94}, the OSC and SOSC are equivalent for the family $(f_i)_{i=1}^N$ of contractive similitudes.

Now let $\nu$ be a Borel probability measure on $\mathbb{R}^q$ with compact support. Let $(p_i)_{i=0}^N$ be a probability vector with $p_i>0$ for all $0\leq i\leq N$. Following
\cite{Bar:88,Las:06}, we call $((f_i)_{i=1}^N,(p_i)_{i=0}^N,\nu)$ a
condensation system. There exist (cf. \cite{Bar:88,Las:06}) a unique
measure $\mu$ and a unique non-empty compact set $K$ satisfying
\begin{equation}\label{attractingmeasure}
\mu=p_0\nu+\sum_{i=1}^Np_i\mu\circ f_i^{-1},\;K={\rm supp}(\nu)\cup\bigg(\bigcup_{i=1}^Nf_i(K)\bigg).
\end{equation}
The measure $\mu$ is called the attracting measure for $((f_i)_{i=1}^N,(p_i)_{i=0}^N,\nu)$ and the set $K(={\rm
supp}(\mu))$ the attractor for this system. In \cite{Olsen:08}, such a measure $\mu$ is also termed an in-homogeneous self-similar measure (ISM); in there, one may find some interesting interpretations for the term "in-homogeneous".
General ISMs may have very complicated behaviors and we can hardly obtain accurate information for the asymptotic quantization errors. We will focus on a particular class of such measures where the measures $\nu$ as involved in (\ref{attractingmeasure}) are self-similar associated with $(f_i)_{i=1}^N$.

\begin{remark}{\rm
Let $p_0=0$. Then $(p_i)_{i=1}^N$ is a probability vector and the measure $\mu$ reduces to a self-similar measure, namely, $\mu$ is the unique Borel probability measure satisfying $\mu=p_1\circ f_1^{-1}+\ldots+p_N\mu\circ f_N^{-1}$.
Assuming the OSC for $(f_i)_{i=1}^N$, Graf and Luschgy proved that $0<\underline{Q}_r^{k_r}(\mu)\leq\overline{Q}_r^{k_r}(\mu)<\infty$ (see \cite{GL:01}),
where the number $k_r$ is given by
\begin{eqnarray*}
\sum_{i=1}^{N}(p_{i}s_i^{r})^{\frac{k_{r}}{k_{r}+r}}=1,\; r>0.
\end{eqnarray*}
}\end{remark}

In \cite{Zhu:08}, the author has studied a class of ISMs, where the measure $\nu$  is the self-similar measure associated with $(f_i)_{i=1}^N$ and a probability vector $(t_i)_{i=1}^N$, with $t_i>0$ for all $1\leq i\leq N$. In this
case, we have that $K=E$ by the uniqueness of the compact set $K$. Assuming the SSC for $(f_i)_{i=1}^N$, we obtained a characterization for the upper and lower quantization dimension for the ISMs $\mu$ \cite{Zhu:08}. Recently, Roychowdhury gave some bounds for these dimensions \cite{RMK:13}. In the present paper, we further prove that the quantization dimension of $\mu$ of order $r$ exists for every $r\in(0,\infty)$, and we determine the exact values of these dimensions. Besides, we will also consider the finiteness and positivity of the upper and lower quantization coefficient in the exact dimension. For every $r>0$, let $\xi_{1,r},\xi_{2,r}$, be implicitly given by
\begin{equation*}
\sum_{i=1}^N(t_is_i^r)^{\frac{\xi_{1,r}}{\xi_{1,r}+r}}=1;\;\;\sum_{i=1}^N(p_is_i^r)^{\frac{\xi_{2,r}}{\xi_{2,r}+r}}=1.
\end{equation*}
Set $\xi_r:=\max\{\xi_{1,r},\xi_{2,r}\}$. As the main result of the paper, we will prove
\begin{theorem}\label{mthm1}
Assume that $(f_i)_{i=1}^N$ satisfies the OSC. Let $\nu$ be the self-similar measures associated with $(f_i)_{i=1}^N$ and $(t_i)_{i=1}^N$. Let $\mu$ be the ISM as defined in (\ref{attractingmeasure}) with $p_i>0$ for all $0\leq i\leq N$. Then
for every $r>0$, $D_r(\mu)$ exists and equals $\xi_r$ and $\underline{Q}_r^{\xi_r}(\mu)>0$. Moreover, we have

(a) if $\xi_{1,r}>\xi_{2,r}$, then $\overline{Q}_r^{\xi_r}(\mu)<\infty$;

(b) if $\xi_{1,r}=\xi_{2,r}$, then we have that $\underline{Q}^{\xi_r}_r(\mu)=\infty$.

\end{theorem}
As an immediate consequence of Theorem \ref{mthm1}, we have
\begin{corollary}
Let $\nu$ be the self-similar measure associated with $(f_i)_{i=1}^N$ and the probability vector $(s_i^{d_0})_{i=1}^N$, where $d_0$ is the unique positive real number satisfying $\sum_{i=1}^Ns_i^{d_0}=1$. Then, we have
\begin{eqnarray*}
D_r(\mu)=d_0,\;\;0<\underline{Q}_r^{d_0}(\mu)\leq\overline{Q}_r^{d_0}(\mu)<\infty.
\end{eqnarray*}
\end{corollary}
\begin{proof}
By \cite[Theorem 3.1]{GL:01}, one sees that $\xi_{1,r}=d_0$. By H\"{o}lder's inequality,
\begin{eqnarray*}
\sum_{i=1}^N(p_is_i^r)^{\frac{d_0}{d_0+r}}\leq\bigg(\sum_{i=1}^Np_i\bigg)^{\frac{d_0}{d_0+r}}
\bigg(\sum_{i=1}^Ns_i^{d_0}\bigg)^{\frac{r}{d_0+r}}=(1-p_0)^{\frac{d_0}{d_0+r}}<1.
\end{eqnarray*}
This implies that $\xi_{1,r}=d_0>\xi_{2,r}$ and the corollary follows by Theorem \ref{mthm1}.
\end{proof}

Let us make some remarks on our main result. The first one is about the comparison between $\xi_{1,r}$ and $\xi_{2,r}$.
\begin{remark}\label{g14}{\rm
In according with different choices of the probability vectors $(t_i)_{i=1}^{N}$ and $(p_i)_{i=0}^N$, we may have $\xi_{1,r}>\xi_{2,r}$, or $\xi_{1,r}\leq\xi_{2,r}$. Indeed, let
\begin{eqnarray*}
s_1=s_2=\cdots=s_N=:c;\; (t_1,t_2,\ldots,t_N)\neq(N^{-1},\ldots,N^{-1}).
\end{eqnarray*}
Then, by \cite[Theorem 3.1]{GL:01}, we have that $\xi_{1,r}=D_r(\nu)<-\log N/\log c$. Hence,
\begin{eqnarray*}
N^{\frac{1}{\xi_{1,r}}}c>1,\;\;{\rm implying}\;\;N^{-\frac{r}{\xi_{1,r}}}c^{-r}<1.
\end{eqnarray*}
Let $0<p_0<1-N^{-\frac{r}{\xi_{1,r}}}c^{-r}$ and $p_i:=N^{-1}(1-p_0)$ for all $1\leq i\leq N$. Then
\begin{eqnarray*}
\sum_{i=1}^N(p_is_i^r)^{\frac{\xi_{1,r}}{\xi_{1,r}+r}}&=&N (N^{-1}(1-p_0)c^r)^{\frac{\xi_{1,r}}{\xi_{1,r}+r}}=N^{\frac{r}{\xi_{1,r}+r}}((1-p_0)c^r)^{\frac{\xi_{1,r}}{\xi_{1,r}+r}}\\
&>&N^{\frac{r}{\xi_{1,r}+r}}(N^{-\frac{r}{\xi_{1,r}}}c^{-r}c^r)^{\frac{\xi_{1,r}}{\xi_{1,r}+r}}=1.
\end{eqnarray*}
Hence, $\xi_{1,r}<\xi_{2,r}$. If $p_0\geq 1-N^{-\frac{r}{\xi_{1,r}}}c^{-r}$, then we have that $\xi_{1,r}\geq\xi_{2,r}$.
}\end{remark}

\begin{remark}{\rm
one can only get some fragmentary conclusions by applying \cite[Theorem 4.14]{GL:00}. In fact, by \cite[Lemma 4.14]{GL:00}, one easily gets
\begin{eqnarray}\label{z14}
e^r_{n,r}(\mu)\left\{\begin{array}{ll}\geq p_0e^r_{n,r}(\nu)+\sum_{i=1}^Np_is_i^re^r_{n,r}(\mu)\\
\leq p_0e^r_{[\frac{n}{N+1}],r}(\nu)+\sum_{i=1}^Np_is_i^re^r_{[\frac{n}{N+1}],r}(\mu)
\end{array}\right..
\end{eqnarray}
For every $s>0$, by the preceding inequality, we have
\begin{eqnarray}\label{g23}
&&\bigg(1-\sum_{i=1}p_is_i^r\bigg)^{-1}\underline{Q}_r^s(\nu)\leq\underline{Q}_r^s(\mu)\leq\overline{Q}_r^s(\mu)\nonumber\\&&\leq p_0(N+1)^{\frac{r}{s}}\overline{Q}_r^s(\nu)+\overline{Q}_r^s(\mu)(N+1)^{\frac{r}{s}}\sum_{i=1}^Np_is_i^r.
\end{eqnarray}
Even if $(N+1)^{\frac{r}{\xi_r}}\sum_{i=1}^Np_is_i^r<1$, we can not get any useful information because $\overline{Q}_r^{\xi_r}(\mu)$ is possibly infinite.
}\end{remark}

\section{Notations and preliminary facts}

For $\Omega:=\{1,\ldots, N\}$, we write
$\Omega_n:=\Omega^n,\Omega^*:=\bigcup_{n=1}^\infty\Omega_n$. We define $|\sigma|:=n$ for $\sigma\in\Omega_n$ and $\sigma|_0=\theta:=$empty word. For
any $\sigma\in\Omega^*$ with $|\sigma|\geq n$, we
write $\sigma|_n:=(\sigma_1,\ldots,\sigma_n)$.
For $1\leq h<n$ and $\sigma=(\sigma_1,\ldots,\sigma_n)\in\Omega_n$, we set
\begin{eqnarray*}
\sigma^{(l)}_{-h}:=(\sigma_{h+1},\ldots,\sigma_n),\;\;\sigma^{(r)}_{-h}:=(\sigma_1,\ldots,\sigma_{n-h});
\end{eqnarray*}
we also define $\sigma^{(l)}_{-n}=\sigma^{(r)}_{-n}:=\theta$ and $\sigma^{(l)}_{0}=\sigma^{(r)}_{0}:=\sigma$.

For $\sigma,\tau\in\Omega^*$, we write
$\sigma\ast\tau:=(\sigma_1,\ldots,\sigma_{|\sigma|},\tau_1,\ldots,\tau_{|\tau|})$. If $\sigma,\tau\in\Omega^*$ and
$|\sigma|<|\tau|,\sigma=\tau|_{|\sigma|}$, then we write $\sigma\prec\tau$ and call $\sigma$ a predecessor of $\tau$. Two words
$\sigma,\tau\in\Omega^*$ are said to be incomparable if we have neither $\sigma\prec\tau$
nor $\tau\prec\sigma$. A finite set $\Gamma\subset\Omega^*$ is
called a finite anti-chain if any two words $\sigma,\tau$ in
$\Gamma$ are incomparable. A finite anti-chain is said to be maximal if
any word $\sigma\in\Omega^{\mathbb{N}}$ has a predecessor in $\Gamma$. For a word $\sigma=(\sigma_1,\ldots,\sigma_n)\in\Omega_n$, we define $f_{\sigma}:=f_{\sigma_1}\circ\cdots\circ
f_{\sigma_n},E_\sigma:=f_\sigma(E)$ and set
\begin{equation*}
s_\sigma:=\prod_{h=1}^ns_{\sigma_h},\;
t_\sigma:=\prod_{h=1}^nt_{\sigma_h},\;p_\sigma:=\prod_{h=1}^np_{\sigma_h}.
\end{equation*}
For the empty word $\theta$, we also define $p_\theta=t_\theta=s_\theta=1$. Without loss of generality, we assume that the diameter of $E$ equals $1$. Then the diameter of $E_\sigma$ is equal to $s_\sigma$ for every $\sigma\in\Omega^*$.

Let $h(\sigma)$ and $h^{(j)}(\sigma),j=1,2$ be as defined in \cite{Zhu:08}. Namely,
\begin{equation*}
h^{(1)}\big((i)\big):=p_0t_is_i^r,\;\;h^{(2)}\big((i)\big):=p_is_i^r,\;\;1\leq i\leq N.
\end{equation*}
Inductively, for $k\geq 2$ and $\sigma\in\Omega_k$, we define
\begin{equation}\label{hsigma}
h^{(1)}({\tau})=h^{(1)}({\tau^{(r)}_{-1}})t_is_i^r+h^{(2)}({\tau^{(r)}_{-1}})p_0t_is_i^r,\;\;
h^{(2)}({\tau})=h^{(2)}({\tau^{(r)}_{-1}})p_is_i^r.
\end{equation}
We finally define $h(\sigma):=h^{(1)}(\sigma)+h^{(2)}(\sigma)$.
These definitions are given according to the behavior of the ISM $\mu$.

For every $\sigma\in\Omega^*$, let $\mu^{(i)}(E_\sigma):=h^{(i)}(\sigma)s_\sigma^{-r},i=1,2$. Then, we have
\begin{eqnarray*}
&&\mu^{(1)}(E_i)=p_0t_i,\;\mu^{(2)}(E_i)=p_i\;\;{\rm for}\;\;1\leq i\leq N;\\
&&\mu^{(1)}(E_\sigma)=\mu^{(1)}(E_{\sigma^{(r)}_{-1}})t_{\sigma_{|\sigma|}}+\mu^{(2)}(E_{\sigma^{(r)}_{-1}})p_0t_{\sigma_{|\sigma|}},\\
&&\mu^{(2)}(E_\sigma)=\mu^{(2)}(E_{\sigma^{(r)}_{-1}})p_{\sigma_{|\sigma|}},\;\;|\sigma|\geq 2.
\end{eqnarray*}
Inductively, One can see that $\mu^{(2)}(E_\sigma)=p_\sigma$ and
\begin{eqnarray}\label{z8}
\mu^{(1)}(E_\sigma)&=&p_0t_\sigma+p_0p_{\sigma_1} t_{\sigma^{(l)}_{-1}}+p_0p_{\sigma|_2}t_{\sigma^{(l)}_{-2}}+\cdots+p_0p_{\sigma^{(r)}_{-1}} t_{\sigma_k}\nonumber\\&=&\sum_{h=0}^{|\sigma|-1}p_0p_{\sigma|_h}t_{\sigma^{(l)}_{-h}},\;\;\sigma=(\sigma_1,\ldots,\sigma_{|\sigma|}).
\end{eqnarray}
By \cite[Lemma 2]{Zhu:08}, for every finite maximal antichain $\Gamma\subset\Omega^*$, we have
\begin{eqnarray}\label{g6}
\mu&=&\sum_{\sigma\in\Gamma}\big(h^{(1)}(\sigma)s_\sigma^{-r}\nu\circ f_\sigma^{-1}+h^{(2)}(\sigma)s_\sigma^{-r}\mu\circ f_\sigma^{-1}\big)\nonumber\\
&=&\sum_{\sigma\in\Gamma}\big(\mu^{(1)}(E_\sigma)\nu\circ f_\sigma^{-1}+\mu^{(2)}(E_\sigma)\mu\circ f_\sigma^{-1}\big).
\end{eqnarray}
Assuming the SSC for the iterated function system $(f_i)_{i=1}^N$, one easily gets
\begin{equation*}
\mu(E_\sigma)=\mu^{(1)}(E_\sigma)+\mu^{(2)}(E_\sigma),\;\;\sigma\in\Omega^*.
\end{equation*}
By applying the results and methods of Graf \cite[Lemma 3.3]{Gr:95}, we obtain this equality under the weaker OSC. That is,
\begin{lemma}\label{s5}
Assume that $(f_i)_{i=1}^N$ satisfies the OSC. Then $\mu(E_\sigma\cap E_\omega)=0$ for every pair $\sigma,\omega$ of incomparable words. As a consequence, we have
\begin{eqnarray}\label{z1}
\mu(E_\sigma)=\mu^{(1)}(E_\sigma)+\mu^{(2)}(E_\sigma)=\sum_{h=0}^{k-1}p_0p_{\sigma|_h}t_{\sigma^{(l)}_{-h}}+p_\sigma.
\end{eqnarray}
\end{lemma}
\begin{proof}
Let $J$ be the same as in \cite[Theorem 3.2]{Gr:95}(see \cite{Schief:94}), namely, $J$ is a nonempty compact set satisfying (A1) $J={\rm cl}({\rm int}(J))$; (A2) ${\rm int}(J)\cap E\neq\emptyset$; (A3) $f_i(J)\subset J$ for all $1\leq j\leq N$; (A4) $f_i({\rm int}(J))\cap f_j({\rm int}(J))=\emptyset$. Here, ${\rm int}(A)$ and ${\rm cl}(A)$ respectively denote the interior and the closure in $\mathbb{R}^q$ of a set $A$. By induction, one can see that $f_\omega({\rm int}(J))\subset{\rm int}(J)$, and for every pair $\sigma, \omega$ of incomparable words, we have $f_\omega({\rm int}(J))\cap f_\sigma({\rm int}(J))=\emptyset$. Hence, using (\ref{g6}), we deduce
\begin{eqnarray}\label{g8}
\mu({\rm int}(J))&\geq&\sum_{\omega\in\Omega_n}\mu(f_\omega({\rm int}(J)))\nonumber\\
&=&\sum_{\omega\in\Omega_n}\sum_{\sigma\in\Omega_n}\big(\mu^{(1)}(E_\sigma)\nu\circ f_\sigma^{-1}+\mu^{(2)}(E_\sigma)\mu\circ f_\sigma^{-1}\big)(f_\omega({\rm int}(J)))\nonumber\\&\geq&\sum_{\omega\in\Omega_n}\big(\mu^{(1)}(E_\omega)\nu\circ f_\omega^{-1}+\mu^{(2)}(E_\omega)\mu\circ f_\omega^{-1}\big)(f_\omega({\rm int}(J)))\nonumber\\&=&\sum_{\omega\in\Omega_n}\mu^{(1)}(E_\sigma)\nu({\rm int}(J))+\sum_{\omega\in\Omega_n}\mu^{(2)}(E_\sigma)\mu({\rm int}(J)).
\end{eqnarray}
Since $\nu$ is a self-similar measure, by \cite[Lemma 3.3]{Gr:95}, $\nu({\rm int}(J))=1$. Note that
\begin{eqnarray*}
&&\sum_{\omega\in\Omega_n}(\mu^{(1)}(E_\sigma)+\mu^{(2)}(E_\sigma))=\sum_{\omega\in\Omega_n}
\big(\sum_{h=0}^{n-1}p_0p_{\sigma|_h}t_{\sigma^{(l)}_{-h}}+p_\omega\big)\\&&=p_0\sum_{h=0}^{n-1}\sum_{\omega\in\Omega_n}p_{\sigma|_h}t_{\sigma^{(l)}_{-h}}+
\sum_{\omega\in\Omega_n}p_\omega\\&&=p_0\sum_{h=0}^{n-1}\big(\sum_{i=1}^Np_i\big)^h\big(\sum_{i=1}^Nt_i\big)^{k-h}+\big(\sum_{i=1}^Np_i\big)^n\\&&=
p_0\sum_{h=0}^{n-1}(1-p_0)^h+(1-p_0)^n=1.
\end{eqnarray*}
Using this, (\ref{g8}) and the fact that $\nu({\rm int}(J))=1$, we get,
$\mu({\rm int}(J))\geq 1$.
It follows that $\mu({\rm int}(J))=\mu(J)=1$ and $\mu(\mathbb{R}^q\setminus{\rm int}(J))=0$. By \cite[p. 227]{Gr:95},
\begin{equation*}
f_\omega^{-1}(J_\sigma)\cap{\rm int}(J)=\emptyset,\;\;{\rm if}\;\;\sigma\nprec\omega\;{\rm and}\;\omega\nprec\sigma.
\end{equation*}
 Hence, $\mu(f_\omega^{-1}(J_\sigma))=0$ for all incomparable pairs $\sigma,\omega$. For such a pair $\sigma,\omega$, let $\Gamma$ be an arbitrary finite maximal antichain containing $\sigma,\tau$. Note that $J_\sigma\cap J_\omega\subset J_\sigma,J_\omega$ Thus, for $\tau\in\Gamma$, we have
\begin{eqnarray}\label{g9}
\mu\circ f_\tau^{-1}(J_\sigma\cap J_\omega)
\left\{ \begin{array}{ll}
\leq\mu\circ f_\tau^{-1}(J_\sigma)=0\;\;\;\;\;\;\;\tau\neq\sigma,\omega\\
\leq\mu\circ f_\tau^{-1}(J_\omega)=0,\;\;\;\;\; \tau=\sigma\\
\leq\mu\circ f_\tau^{-1}(J_\sigma)=0\;\;\;\;\;\;\; \tau=\omega
\end{array}\right..
\end{eqnarray}
Again, note that $\nu$ is a self-similar measure. By \cite[Lemma 3.3]{Gr:95}, for all $\tau\in\Gamma$, we have, $\nu\circ f_\tau^{-1}(J_\sigma\cap J_\omega)=0$ (for the same reason as (\ref{g9})).
This, together with (\ref{g6}) and (\ref{g9}), yields
\begin{eqnarray*}
\mu(J_\sigma\cap J_\omega)=\sum_{\tau\in\Gamma}\big(\mu^{(1)}(E_\tau)\nu\circ f_\tau^{-1}+\mu^{(2)}(E_\tau)\mu\circ f_\tau^{-1}\big)(J_\sigma\cap J_\omega)=0.
\end{eqnarray*}
Since $E_\sigma\cap E_\tau\subset J_\sigma\cap J_\omega$, we conclude that $\mu(E_\sigma\cap E_\omega)=0$. Finally, by (\ref{g6}),
\begin{eqnarray*}
&&\mu(E_\sigma)=\sum_{\tau\in\Omega_{|\sigma|}}(\mu^{(1)}(E_\tau)\nu\circ f_\tau^{-1}+\mu^{(2)}(E_\tau)\mu\circ f_\tau^{-1})(E_\sigma)\\&&=\mu^{(1)}(E_\sigma)\nu\circ f_\sigma^{-1}(E_\sigma)
+\mu^{(2)}(E_\sigma)\mu\circ f_\sigma^{-1}(E_\sigma)
=\sum_{h=0}^{k-1}p_0p_{\sigma|_h}t_{\sigma^{(l)}_{-h}}+p_\sigma.
\end{eqnarray*}
This completes the proof of the lemma.
\end{proof}

Let $\underline{\eta}_r:=\min_{1\leq i\leq N}\min\big\{(p_0t_i+p_i),t_i\big\}s_i^r$. Then, we have
\begin{eqnarray*}
h(\sigma)\geq \underline{\eta}_r h(\sigma^{(r)}_{-1})\;\;{\rm for\;all}\;\;\sigma\in\Omega^*.
\end{eqnarray*}
Inspired by \cite[(14.5)]{GL:01}, we define the following finite maximal antichains:
\begin{equation}\label{gamman}
\Lambda_{k,r}:=\big\{\sigma\in\Omega^*:\;h(\sigma^{(r)}_{-1})\geq
k^{-1}\underline{\eta}_r>h(\sigma)\big\},\;\;k\geq 1.
\end{equation}
We denote by $N_{k,r}$ the cardinality of $\Lambda_{k,r}$. Assume the SSC for $(f_i)_{i=1}^N$, we have essentially proved the following estimates \cite[Lemmas 7,10]{Zhu:08}:
\begin{eqnarray}\label{knownestimate}
D\sum_{\sigma\in\Lambda_{k,r}}h(\sigma)\leq e^r_{ N_{k,r},r}(\mu)\leq\sum_{\sigma\in\Lambda_{k,r}}h(\sigma).
\end{eqnarray}
where $D>0$ is a constant independent of $k$. Next, we show that these estimates are valid if we assume the OSC for $(f_i)_{i=1}^N$. That is,
\begin{lemma}\label{errorestimate}
Assume that $(f_i)_{i=1}^N$ satisfies the OSC. Then (\ref{knownestimate}) holds.
\end{lemma}
\begin{proof}
For each $\sigma\in\Lambda_{k,r}$, let $a_\sigma$ be an arbitrary point of $E_\sigma$. Then
\begin{eqnarray*}
e^r_{ N_{k,r},r}(\mu)\leq\sum_{\sigma\in\Lambda_{k,r}}\int_{E_\sigma}d(x,a_\sigma)^rd\mu(x)\leq\sum_{\sigma\in\Lambda_{k,r}}\mu(E_\sigma)c_\sigma^r
=\sum_{\sigma\in\Lambda_{k,r}}h(\sigma).
\end{eqnarray*}
 So, it suffices to show the first inequality. By \cite{Schief:94}, there exists a non-empty open set $U$ with $U\cap E\neq\emptyset$ such that
$f_i(U)\subset U$ for all $i=1,2,\ldots,N$, and $f_i(U)\cap f_j(U)=\emptyset$ for $i\neq j$. By induction, for every pair of incomparable words $\sigma,\tau\in\Omega^*$, we have, $f_\sigma(U),f_\tau(U)\subset U$ and $f_\sigma(U)\cap f_\tau(U)=\emptyset$.
 Thus, we may fix a word $\tau\in\Omega^*$ such that $E_\tau\subset U$ and $E_{\sigma\ast\tau},\sigma\in\Lambda_{k,r}$, are pairwise disjoint. Since $E_\tau$ is a compact set, we have that $\delta:=d(E_\tau, U^c)>0$. Note that $f_\sigma(E_{\sigma\ast\tau})\subset f_\sigma(U)$ for every $\sigma\in\Omega^*$. Hence, for two distinct words $\sigma,\omega\in\Lambda_{k,r}$, we have
\begin{eqnarray*}
d(E_{\sigma\ast\tau},E_{\omega\ast\tau})&\geq&\max\{d(E_{\sigma\ast\tau},f_\sigma(U)^c),d(E_{\omega\ast\tau},f_\omega(U)^c)\}\\
&=&\max\{d(E_{\sigma\ast\tau},f_\sigma(U^c)),d(E_{\omega\ast\tau},f_\omega(U^c))\}\\&=&\delta\max\{s_\sigma,s_\omega\}
\end{eqnarray*}
This allows us to apply the arguments in \cite[Lemma 10]{Zhu:08} to the measure $\lambda$:
\begin{eqnarray*}
\lambda:=\mu\big(\cdot|\bigcup_{\sigma\in\Lambda_{k,r}}E_{\sigma\ast\tau}\big):=
\frac{\mu\big(\cdot\cap\big(\bigcup_{\sigma\in\Lambda_{k,r}}E_{\sigma\ast\tau}\big)\big)}{\mu\big(\bigcup_{\sigma\in\Lambda_{k,r}}E_{\sigma\ast\tau}\big)},
\end{eqnarray*}
 and one may find a constant $A_1$, which is independent of $k$, such that
\begin{eqnarray}\label{z10}
e^r_{ N_{k,r},r}(\lambda)\geq A_1\sum_{\sigma\in\Lambda_{k,r}}h(\sigma\ast\tau)=A_1\underline{\eta}_r^{|\tau|}\sum_{\sigma\in\Lambda_{k,r}}h(\sigma).
\end{eqnarray}
Since $E_{\sigma\ast\tau},\sigma\in\Lambda_{k,r}$, are pairwise disjoint, we have
 \begin{eqnarray*}
 \mu\big(\bigcup_{\sigma\in\Lambda_{k,r}}E_{\sigma\ast\tau}\big)&=&\sum_{\sigma\in\Lambda_{k,r}}h(\sigma\ast\tau)s_{\sigma\ast\tau}^{-r}\geq
 \underline{\eta}_r^{|\tau|}s_\tau^{-r}\sum_{\sigma\in\Lambda_{k,r}}h(\sigma)s_\sigma^{-r}\\&\geq&
 \underline{\eta}_r^{|\tau|}s_\tau^{-r}\sum_{\sigma\in\Lambda_{k,r}}\mu(E_\sigma)\geq s_\tau^{-r}\underline{\eta}_r^{|\tau|}.
 \end{eqnarray*}
Using this and (\ref{z10}), we further deduce
\begin{eqnarray*}
e^r_{ N_{k,r},r}(\mu)\geq \mu\big(\bigcup_{\sigma\in\Lambda_{k,r}}E_{\sigma\ast\tau}\big)e^r_{ N_{k,r},r}(\lambda)\geq A_1\underline{\eta}_r^{|\tau|} s_\tau^{-r}\underline{\eta}_r^{|\tau|}\sum_{\sigma\in\Lambda_{k,r}}h(\sigma).
\end{eqnarray*}
Hence, the proof of the lemma is complete by setting $D:= A_1\underline{\eta}_r^{2|\tau|}$.
\end{proof}

\section{Proof of Theorem \ref{mthm1}}
To give the proof for Theorem \ref{mthm1} (2), we need to establish several lemmas.
For every $r>0$ and $s>0$, we write
\begin{eqnarray}\label{abs}
a(s):=\sum_{i=1}^N(t_is_i^r)^{\frac{s}{s+r}},\;\;b(s):=\sum_{i=1}^N(p_is_i^r)^{\frac{s}{s+r}}.
\end{eqnarray}
\begin{lemma}\label{s0}
For every $k\geq 1$, we have
\begin{eqnarray*}
p_0^{\frac{s}{s+r}}\max\{a(s)^k,b(s)^k\}\leq\sum_{\sigma\in\Omega_k}h(\sigma)^{\frac{s}{s+r}}\leq p_0^{\frac{s}{s+r}}\sum_{h=0}^{k-1}a(s)^hb(s)^{k-h}+b(s)^k.
\end{eqnarray*}
\end{lemma}
\begin{proof}
For every $\sigma=(\sigma_1,\ldots,\sigma_k)\in\Omega_k$, by (\ref{z1}), we have
\begin{eqnarray}\label{z2}
&&\mu(E_\sigma)^{\frac{s}{s+r}}=\big(\sum_{h=0}^{k-1}p_0p_{\sigma|_h}t_{\sigma^{(l)}_{-h}}+p_\sigma\big)^{\frac{s}{s+r}}\nonumber
\\&&\;\;\;\;\;\;\;\;\left\{\begin{array}{ll}\leq\sum_{h=0}^{k-1}(p_0p_{\sigma|_h}t_{\sigma^{(l)}_{-h}})^{\frac{s}{s+r}}+p_\sigma^{\frac{s}{s+r}}\\
\geq\max\{(p_0t_\sigma)^{\frac{s}{s+r}},p_\sigma^{\frac{s}{s+r}}\}
\end{array}\right..
\end{eqnarray}
Note that $h(\sigma)=\mu(E_\sigma)s_\sigma^r$ for $\sigma\in\Omega^*$. It follows that
\begin{eqnarray*}
&&\sum_{\sigma\in\Omega_k}h(\sigma)^{\frac{s}{s+r}}=\sum_{\sigma\in\Omega_k}(\mu(E_\sigma)s_\sigma^r)^{\frac{s}{s+r}}\\&&
\leq\sum_{\sigma\in\Omega_k}\sum_{h=0}^{k-1}(p_0p_{\sigma|_h}t_{\sigma^{(l)}_{-h}}s_\sigma^r)^{\frac{s}{s+r}}+\sum_{\sigma\in\Omega_k}(p_\sigma s_\sigma^r)^{\frac{s}{s+r}}\\
&&=\sum_{h=0}^{k-1}\sum_{\sigma\in\Omega_k}(p_0p_{\sigma|_h}t_{\sigma^{(l)}_{-h}}s_\sigma^r)^{\frac{s}{s+r}}+\sum_{\sigma\in\Omega_k}(p_\sigma s_\sigma^r)^{\frac{s}{s+r}}
\\&&=p_0^{\frac{s}{s+r}}\sum_{h=0}^{k-1}a(s)^{k-h}b(s)^h+b(s)^k.
\end{eqnarray*}
Using the second inequality in (\ref{z2}), one easily gets the remaining inequality.
\end{proof}

As Lemma \ref{errorestimate} shows, the quantization errors are characterized
in terms of the finite maximal antichains as defined in (\ref{gamman}); however, these antichains typically consist of words of different length. In the following, we need to estimate the sums over these antichains by means of words of same length. Similar situations often occur in the study of fractal geometry, one may see \cite{HL:96} for example. For $k\geq 1$, we define
\begin{eqnarray*}
l_{1k}:=\min_{\sigma\in\Lambda_{k,r}}|\sigma|,\;\;l_{2k}:=\max_{\sigma\in\Lambda_{k,r}}|\sigma|.
\end{eqnarray*}

\begin{lemma}\label{s1}
For $a(s), b(s)$ as defined in (\ref{abs}), we have
\begin{eqnarray}\label{z4'}
\left\{\begin{array}{ll}a(s)^{l_{1k}}\leq\sum_{\sigma\in\Lambda_{k,r}}(t_\sigma s_\sigma^r)^{\frac{s}{s+r}}\leq a(s)^{l_{2k}},\;{\rm if}\;\;s\leq\xi_{1,r} \\
a(s)^{l_{2k}}\leq\sum_{\sigma\in\Lambda_{k,r}}(t_\sigma s_\sigma^r)^{\frac{s}{s+r}}\leq a(s)^{l_{1k}},\;{\rm if}\;\;s\geq \xi_{1,r}\end{array}\right..
\end{eqnarray}
With $b(s), p_\sigma,\xi_{2,r}$ in place of $a(s), t_\sigma,\xi_{1,r}$, (\ref{z4'}) remains true.
\end{lemma}
\begin{proof}
As we did in \cite[Lemma 2.5]{Zhu:12}, for every $s>0$ and $k\geq 1$, we write
\begin{eqnarray*}
I_k(s):=\sum_{\sigma\in\Omega_k}(t_\sigma s_\sigma^r)^{\frac{s}{s+r}};\;\;\xi(\sigma):=I^{-1}_k(s)(t_\sigma s_\sigma^r)^{\frac{s}{s+r}},\;\;\sigma\in\Omega_k.
\end{eqnarray*}
For $\sigma\in\Omega_k$ and $h\geq 0$, we set $\Lambda_h(\sigma):=\{\tau\in\Omega_{k+h}:\sigma\prec\tau\}$. Then, we have
\begin{eqnarray*}
\sum_{\tau\in\Lambda_h(\sigma)}\xi(\tau)&=&I_{k+h}(s)^{-1}(t_\sigma s_\sigma^r)^{\frac{s}{s+r}}\sum_{\omega\in\Omega_h} (t_\omega s_\omega^r)^{\frac{s}{s+r}}\\&=&I_{k+h}(s)^{-1}(t_\sigma s_\sigma^r)^{\frac{s}{s+r}}I_{k+h}(s)I_k(s)^{-1}=\xi(\sigma).
\end{eqnarray*}
It follows that, for every $k\geq 1$,
\begin{eqnarray*}
\sum_{\sigma\in\Lambda_{k,r}}I^{-1}_{|\sigma|}(s)(t_\sigma s_\sigma^r)^{\frac{s}{s+r}}=\sum_{\sigma\in\Lambda_{k,r}}\xi(\sigma)=\sum_{\sigma\in\Omega_k}\xi(\sigma)=1.
\end{eqnarray*}
As an immediate consequence, we obtain
\begin{eqnarray}\label{z4}
\min_{l_{1k}\leq h\leq l_{2k}}I_k(s)\leq\sum_{\sigma\in\Lambda_{k,r}}(t_\sigma s_\sigma^r)^{\frac{s}{s+r}}\leq\max_{l_{1k}\leq h\leq l_{2k}}I_k(s).
\end{eqnarray}
Note that, $a(s)>1$ if $s<\xi_{1,r}$ and $a(s)\leq1$ if $s\geq \xi_{1,r}$. So $I_h(s)=a(s)^h$ is increasing with respect to $h$ for all $s<\xi_{1,r}$ and it is decreasing when $s\geq \xi_{1,r}$. This and (\ref{z4}) imply (\ref{z4'}).
One can show the remaining assertions similarly.
\end{proof}
\begin{lemma}
For every $k\in\mathbb{N}$, we have
\begin{eqnarray}\label{z5}
\sum_{\sigma\in\Lambda_{k,r}}(h(\sigma))^{\frac{s}{s+r}}\left\{\begin{array}{ll}\geq p_0^{\frac{s}{s+r}}a(s)^{l_{1k}},\;{\rm if}\;\;s\leq \xi_{1,r} \\
\geq b(s)^{l_{1k}},\;\;\;\;\;\;\;\;{\rm if}\;\;s\leq \xi_{2,r}\end{array}\right..
\end{eqnarray}
\end{lemma}
\begin{proof}
Let $s\leq\xi_{1,r}$ be given. By (\ref{z1}) and (\ref{z4'}), we have
\begin{eqnarray*}
\sum_{\sigma\in\Lambda_{k,r}}(h(\sigma))^{\frac{s}{s+r}}\geq\sum_{\sigma\in\Lambda_{k,r}}(p_0t_\sigma s_\sigma^r)^{\frac{s}{s+r}}=p_0^{\frac{s}{s+r}}\sum_{\sigma\in\Lambda_{k,r}}(t_\sigma s_\sigma^r)^{\frac{s}{s+r}}\geq p_0^{\frac{s}{s+r}}a(s)^{l_{1k}}.
\end{eqnarray*}
One can see the remaining inequality in a similar manner.
\end{proof}
\begin{lemma}\label{s3}
$s\geq\xi_r=\max\{\xi_{1,r},\xi_{2,r}\}$ and $k\in\mathbb{N}$, we have
\begin{eqnarray*}
\sum_{\sigma\in\Lambda_{k,r}}h(\sigma)^{\frac{s}{s+r}}\leq\sum_{h=0}^{l_{1k}-1}a(s)^{l_{1k}-h}b(s)^{h}+\sum_{h=l_{1k}}^{l_{2k}}b(s)^h+b(s)^{l_{1k}}.
\end{eqnarray*}
\end{lemma}
\begin{proof}
For $\sigma\in\Omega_{l_{1k}}$ and $h\geq l_{1,k}$, we write
\begin{eqnarray*}
\Gamma_k(\sigma):=\{\tau\in\Lambda_{k,r}:\sigma\prec\tau\},\;\;l_k(\sigma):=\max_{\tau\in\Gamma_k(\sigma)}|\tau|,\;
\Gamma_{k,h}(\sigma):=\Gamma_k(\sigma)\cap\Omega_h.
\end{eqnarray*}
Note that $\Lambda_{k,r}$ is a finite maximal anti-chain. For every $\tau\in\Gamma_{k,l_k(\sigma)}(\sigma)$,
we have that $\tau^{(r)}_{-1}\ast i\in\Lambda_{k,r}$ for all $1\leq i\leq N$. By (\ref{hsigma}), we deduce
\begin{eqnarray}\label{g1}
&&\sum_{i=1}^Nh^{(1)}(\tau^{(r)}_{-1}\ast i)^{\frac{s}{s+r}}=\sum_{i=1}^N\big(h^{(1)}(\tau^{(r)}_{-1})t_is_i^r+h^{(2)}(\tau^{(r)}_{-1})p_0t_is_i^r\big)^{\frac{s}{s+r}}\nonumber\\
&&\leq\sum_{i=1}^N\big(h^{(1)}(\tau^{(r)}_{-1})t_is_i^r\big)^{\frac{s}{s+r}}+\sum_{i=1}^N\big(h^{(2)}(\tau^{(r)}_{-1})p_0t_is_i^r\big)^{\frac{s}{s+r}}\\&&=
h^{(1)}(\tau^{(r)}_{-1})^{\frac{s}{s+r}}\sum_{i=1}^N(t_is_i^r)^{\frac{s}{s+r}}+h^{(2)}(\tau^{(r)}_{-1})^{\frac{s}{s+r}}\sum_{i=1}^N(p_0t_is_i^r)^{\frac{s}{s+r}}
\nonumber\\&&\leq h^{(1)}(\tau^{(r)}_{-1})^{\frac{s}{s+r}}+h^{(2)}(\tau^{(r)}_{-1})^{\frac{s}{s+r}},\label{g4}
\end{eqnarray}
where we used the fact that $a(s)\leq 1$ for $s\geq\xi_r$.
 For every $\tau\in\Gamma_{k,l_k(\sigma)}(\sigma)$, we have that $\tau^{(r)}_{-h}\notin\Lambda_{k,r}$ for any $1\leq h\leq |\tau|-1$. So, for every $1\leq i\leq N$, with $\tau^{(r)}_{-2}\ast i\neq\tau^{(r)}_{-1}$, we have exactly the following two possible cases:
\begin{eqnarray}\label{g2}
\tau^{(r)}_{-2}\ast i\in\Lambda_{k,r},\;\;{\rm or}\;\;\tau^{(r)}_{-2}\ast i\ast j\in\Lambda_{k,r} \;{\rm for\; all}\;\; j=1,2\ldots,N.
\end{eqnarray}
In both cases, the contribution of $\rho^{(i)}:=\tau^{(r)}_{-2}\ast i$ to the sum
\[
\sum_{\tau\in\Gamma_k(\sigma)}h^{(1)}(\tau)^{\frac{\xi_r}{\xi_r+r}}
\]
is not greater than $ h^{(1)}(\rho^{(i)})^{\frac{s}{s+r}}+h^{(2)}(\rho^{(i)})^{\frac{s}{s+r}}$. We write
\begin{eqnarray} \label{g5} &&\widetilde{\Gamma}_k^{(1)}(\sigma):=(\Gamma_k(\sigma)\setminus\Omega_{l_k(\sigma)})\cup\{\tau^{(r)}_{-1}:\tau\in\Gamma_{k,l_k(\sigma)}(\sigma)\};\\
&&\widetilde{\Gamma}_{k,l_k(\sigma)-1}^{(1)}(\sigma):=\widetilde{\Gamma}_k^{(1)}(\sigma)\cap\Omega_{l_k(\sigma)-1}\nonumber.
 \end{eqnarray}
 By (\ref{g4}) and (\ref{g2}), we deduce
\begin{eqnarray*}
\sum_{\tau\in\Gamma_k(\sigma)}h^{(1)}(\tau)^{\frac{\xi_r}{\xi_r+r}}\leq\sum_{\tau\in\widetilde{\Gamma}^{(1)}_k(\sigma)}h^{(1)}(\tau)^{\frac{\xi_r}{\xi_r+r}}
+\sum_{\tau\in\widetilde{\Gamma}_{k,l_k(\sigma)-1}^{(1)}(\sigma)}h^{(2)}(\tau)^{\frac{s}{s+r}}.
\end{eqnarray*}
Again, using (\ref{hsigma}), for $\tau\in\widetilde{\Gamma}_{k,l_k(\sigma)-1}^{(1)}(\sigma)$, (\ref{g4}) holds. Write
\begin{eqnarray}\label{g11}
&&\widetilde{\Gamma}_k^{(2)}(\sigma):=(\widetilde{\Gamma}_k^{(1)}(\sigma)\setminus\Omega_{l_k(\sigma)-1})
\cup\{\tau^{(r)}_{-1}:\tau\in\widetilde{\Gamma}_{k,l_k(\sigma)-1}^{(1)}(\sigma)\}\\
&&\widetilde{\Gamma}_{k,l_k(\sigma)-2}^{(2)}(\sigma):=\widetilde{\Gamma}_k^{(2)}(\sigma)\cap\Omega_{l_k(\sigma)-2}.\nonumber
\end{eqnarray}
 Then, in a similar manner, one can see
\begin{eqnarray*}
\sum_{\tau\in\Gamma_k(\sigma)}h^{(1)}(\tau)^{\frac{\xi_r}{\xi_r+r}}\leq\sum_{\tau\in\widetilde{\Gamma}^{(2)}_k(\sigma)}h^{(1)}(\tau)^{\frac{\xi_r}{\xi_r+r}}
+\sum_{i=1}^2\sum_{\tau\in\widetilde{\Gamma}_{k,l_k(\sigma)-i}^{(i)}(\sigma)}h^{(2)}(\tau)^{\frac{s}{s+r}}.
\end{eqnarray*}
By repeating the above estimate finitely many times, we obtain
\begin{eqnarray*}
\sum_{\tau\in\Gamma_k(\sigma)}h^{(1)}(\tau)^{\frac{s}{s+r}}\leq h^{(1)}(\sigma)^{\frac{s}{s+r}}+\sum_{h=0}^{l_k(\sigma)-l_{1k}}\sum_{\tau\in\Lambda_h(\sigma)}h^{(2)}(\tau)^{\frac{s}{s+r}}.
\end{eqnarray*}
Note that the preceding argument holds for an arbitrary word in $\Omega_{1k}$. Thus,
\begin{eqnarray*}
&&\sum_{\sigma\in\Lambda_{k,r}}h^{(1)}(\sigma)^{\frac{s}{s+r}}=\sum_{\sigma\in\Omega_{l_{1k}}}\sum_{\tau\in\Gamma_k(\sigma)}h^{(1)}(\tau)^{\frac{s}{s+r}}
\\&&\leq\sum_{\sigma\in\Omega_{l_{1k}}}h^{(1)}(\sigma)^{\frac{s}{s+r}}+
\sum_{\sigma\in\Omega_{l_{1k}}}\sum_{h=0}^{l_k(\sigma)-l_{1k}}\sum_{\tau\in\Lambda_h(\sigma)}h^{(2)}(\tau)^{\frac{s}{s+r}}
\end{eqnarray*}
As $l_k(\sigma)\leq l_{2k}$ for all $\sigma\in\Omega_{l_{1k}}$, by Lemma \ref{s0}, we further deduce
\begin{eqnarray}\label{z3}
\sum_{\sigma\in\Lambda_{k,r}}h^{(1)}(\sigma)^{\frac{s}{s+r}}&\leq&\sum_{\sigma\in\Omega_{l_{1k}}}h^{(1)}(\sigma)^{\frac{s}{s+r}}+
\sum_{\sigma\in\Omega_{l_{1k}}}\sum_{h=0}^{l_{2k}-l_{1k}}\sum_{\tau\in\Lambda_h(\sigma)}h^{(2)}(\tau)^{\frac{s}{s+r}}\nonumber\\
&=&\sum_{\sigma\in\Omega_{l_{1k}}}h^{(1)}(\sigma)^{\frac{s}{s+r}}+
\sum_{h=l_{1k}}^{l_{2k}}\sum_{\tau\in\Omega_h(\sigma)}h^{(2)}(\tau)^{\frac{s}{s+r}}\nonumber\\
&\leq&\sum_{h=0}^{l_{1k}-1}a(s)^{l_{1k}-h}b(s)^{h}+\sum_{h=l_{1k}}^{l_{2k}}b(s)^h.
\end{eqnarray}
Observe that $b(s)\leq 1$ for all $s\geq\xi_r$. By Lemma \ref{s1}, we have
\begin{eqnarray*}
\sum_{\sigma\in\Lambda_{k,r}}h^{(2)}(\sigma)^{\frac{s}{s+r}}=\sum_{\sigma\in\Lambda_{k,r}}(p_\sigma s_\sigma)^{\frac{s}{s+r}}\leq\sum_{\sigma\in\Omega_{l_{1k}}}(p_\sigma s_\sigma)^{\frac{s}{s+r}}=b(s)^{l_{1k}}.
\end{eqnarray*}
This, together with (\ref{z3}), yields
\begin{eqnarray*}
\sum_{\sigma\in\Lambda_{k,r}}h(\sigma)^{\frac{s}{s+r}}&\leq&\sum_{\sigma\in\Lambda_{k,r}}h^{(1)}(\sigma)^{\frac{s}{s+r}}+
\sum_{\sigma\in\Lambda_{k,r}}h^{(2)}(\sigma)^{\frac{s}{s+r}}\\
&\leq&\sum_{h=0}^{l_{1k}-1}a(s)^{l_{1k}-h}b(s)^{h}+\sum_{h=l_{1k}}^{l_{2k}}b(s)^h+b(s)^{l_{1k}}.
\end{eqnarray*}
This completes the proof of the lemma.
\end{proof}
\begin{lemma}\label{s4}
Assume that $\xi_r=\xi_{1,r}=\xi_{2,r}$. Then we have
\begin{eqnarray*}
\sum_{\sigma\in\Lambda_{k,r}}h(\sigma)^{\frac{\xi_r}{\xi_r+r}}\geq (p_0l_{1k})^{\frac{\xi_r}{\xi_r+r}}.
\end{eqnarray*}
\end{lemma}
\begin{proof}
By the hypothesis, we have that $a(\xi_r)=b(\xi_r)=1$. As we noted in the proof of Lemma \ref{s3}, for every $\tau\in\Gamma_{k,l_k(\sigma)}(\sigma)$ and $1\leq i\leq N$, we have, $\tau^{(r)}_{-1}\ast i\in\Lambda_{k,r}$.
In contrast to (\ref{g1}), we have
\begin{eqnarray}\label{g12}
\sum_{i=1}^Nh^{(1)}(\tau^{(r)}_{-1}\ast i)^{\frac{\xi_r}{\xi_r+r}}&=&\sum_{i=1}^N\big(h^{(1)}(\tau^{(r)}_{-1})t_is_i^r+h^{(2)}(\tau^{(r)}_{-1})p_0t_is_i^r\big)^{\frac{\xi_r}{\xi_r+r}}\nonumber\\
&\geq&\sum_{i=1}^N\big(h^{(1)}(\tau^{(r)}_{-1})t_is_i^r\big)^{\frac{\xi_r}{\xi_r+r}}\nonumber\\&=&
h^{(1)}(\tau^{(r)}_{-1})^{\frac{\xi_r}{\xi_r+r}}a(\xi_r)=h^{(1)}(\tau^{(r)}_{-1})^{\frac{\xi_r}{\xi_r+r}}.
\end{eqnarray}
Let $\widetilde{\Gamma}^{(i)}_k(\sigma),i=1,2$, be as defined in (\ref{g5}) and (\ref{g11}). By (\ref{g2}) and (\ref{g12}),
\begin{eqnarray*}
\sum_{\tau\in\Gamma_k(\sigma)}h^{(1)}(\tau)^{\frac{\xi_r}{\xi_r+r}}\geq\sum_{\tau\in\widetilde{\Gamma}^{(1)}_k(\sigma)}h^{(1)}(\tau)^{\frac{\xi_r}{\xi_r+r}}
\geq\sum_{\tau\in\widetilde{\Gamma}^{(2)}_k(\sigma)}h^{(1)}(\tau)^{\frac{\xi_r}{\xi_r+r}}.
\end{eqnarray*}
Repeating the above process finitely many times, we obtain
\begin{eqnarray*}
\sum_{\tau\in\Gamma_k(\sigma)}h^{(1)}(\tau)^{\frac{\xi_r}{\xi_r+r}}\geq h^{(1)}(\sigma)^{\frac{\xi_r}{\xi_r+r}}.
\end{eqnarray*}
We apply the preceding inequality to all $\sigma\in\Lambda_{k,r}$. Then we have
\begin{eqnarray}\label{g3}
\sum_{\tau\in\Lambda_{k,r}}h^{(1)}(\tau)^{\frac{\xi_r}{\xi_r+r}}\geq\sum_{\sigma\in\Omega_{l_{1k}}}h^{(1)}(\sigma)^{\frac{\xi_r}{\xi_r+r}}
\end{eqnarray}
For every $\sigma\in\Omega_{l_{1k}}$, using (\ref{z8}) and H\"{o}lder's inequality (with exponent less than one), we deduce
\begin{eqnarray*}
h^{(1)}(\sigma)=p_0\sum_{h=0}^{l_{1k}-1}p_{\sigma|_h}t_{\sigma^{(l)}_{-h}}s_\sigma^r
\geq p_0\bigg(\sum_{h=0}^{l_{1k}-1}(p_{\sigma|_h}t_{\sigma^{(l)}_{-h}}s_\sigma^r)^{\frac{\xi_r}{\xi_r+r}}\bigg)^{\frac{\xi_r+r}{\xi_r}}\cdot l_{1k}^{-\frac{r}{\xi_r}}.
\end{eqnarray*}
As an immediate consequence, we obtain
\begin{eqnarray*}
h(\sigma)^{\frac{\xi_r}{\xi_r+r}}\geq h^{(1)}(\sigma)^{\frac{\xi_r}{\xi_r+r}}\geq p_0^{\frac{\xi_r}{\xi_r+r}}
\sum_{h=0}^{l_{1k}-1}(t_{\sigma^{(l)}_{-h}}p_{\sigma|_h}s_\sigma^r)^{\frac{\xi_r}{\xi_r+r}}\cdot l_{1k}^{-\frac{r}{\xi_r+r}}.
\end{eqnarray*}
Hence, by (\ref{g3}), it follows that
\begin{eqnarray*}
\sum_{\sigma\in\Lambda_{k,r}}h(\sigma)^{\frac{\xi_r}{\xi_r+r}}&\geq&\sum_{\sigma\in\Omega_{l_{1k}}}p_0^{\frac{\xi_r}{\xi_r+r}}
\sum_{h=0}^{l_{1k}-1}(p_{\sigma|_h}t_{\sigma^{(l)}_{-h}}s_\sigma^r)^{\frac{\xi_r}{\xi_r+r}}\cdot l_{1k}^{-\frac{r}{\xi_r+r}}\\
&=&p_0^{\frac{\xi_r}{\xi_r+r}}l_{1k}^{-\frac{r}{\xi_r+r}}\sum_{h=0}^{l_{1k}-1}
\sum_{\sigma\in\Omega_{l_{1k}}}(p_{\sigma|_h}t_{\sigma^{(l)}_{-h}}s_\sigma^r)^{\frac{\xi_r}{\xi_r+r}}
\\&=&p_0^{\frac{\xi_r}{\xi_r+r}}l_{1k}^{-\frac{r}{\xi_r+r}}
\sum_{h=0}^{l_{1k}-1}a(\xi_r)^{l_{1k}-h}b(\xi_r)^h\\&=&p_0^{\frac{\xi_r}{\xi_r+r}}l_{1k}^{1-\frac{r}{\xi_r+r}}
=(p_0l_{1k})^{\frac{\xi_r}{\xi_r+r}}.
\end{eqnarray*}
This completes the proof of the Lemma.
\end{proof}

Recall that $ N_{k,r}$ denotes the cardinality of $\Lambda_{k,r}$. As in \cite{Zhu:12}, we write
\begin{eqnarray*}
\underline{P}_r^s(\mu):=\liminf_{k\to\infty} N_{k,r}^{\frac{1}{s}}e_{ N_{k,r},r}(\mu);\;\;
\overline{P}_r^s(\mu):=\limsup_{k\to\infty} N_{k,r}^{\frac{1}{s}}e_{ N_{k,r},r}(\mu)
\end{eqnarray*}
\begin{lemma}\label{g13}
$\underline{Q}_r^{s}(\mu)>0$ iff $\underline{P}_r^s(\mu)>0$ and $\underline{Q}_r^{s}(\mu)<\infty$ iff $\underline{P}_r^s(\mu)<\infty$.
The same is true if we replace $\underline{Q}_r^{s}(\mu),\underline{P}_r^s(\mu)$ with $\overline{Q}_r^{s}(\mu),\overline{P}_r^s(\mu)$.
\end{lemma}
\begin{proof}
For every $1\leq i\leq N$, we have
\begin{eqnarray}\label{g10}
h(\sigma\ast i)&=&h^{(1)}(\sigma)t_is_i^r+h^{(2)}(\sigma)p_0t_is_i^r+h^{(2)}(\sigma)p_is_i^r\nonumber\\&\leq&
h(\sigma)\max_{1\leq j\leq N}\{t_j,p_0t_j+p_j\}s_j^r.
\end{eqnarray}
Let $\overline{\eta}_r:=\max_{1\leq j\leq N}\{t_j,p_0t_j+p_j\}s_j^r$. For all $k>(1-\overline{\eta}_r)^{-1}-1$ and every $\sigma\in\Lambda_k$, by (\ref{g10}), we have,
\begin{eqnarray*}
h(\sigma\ast i)\leq h(\sigma)\overline{\eta}_r<(k^{-1}\underline{\eta}_r)\overline{\eta}_r<(k+1)^{-1}\underline{\eta}_r.
\end{eqnarray*}
It follows that, $ N_{k,r}\leq \phi_{k+1,r}\leq N  N_{k,r}$. So, (2.9),(2.10) of \cite{Zhu:12} hold:
\begin{eqnarray*}
\underline{P}_r^s(\mu)\geq\underline{Q}_r^s(\mu)\geq N^{-\frac{1}{s}}\underline{P}_r^s(\mu),\;
\overline{P}_r^s(\mu)\leq\overline{Q}_r^s(\mu)\leq N^{\frac{1}{s}}\overline{P}_r^s(\mu).
\end{eqnarray*}
This completes the proof of the lemma.
\end{proof}
\emph{Proof of Theorem \ref{mthm1}}

For $\xi_r=\max\{\xi_{1,r},\xi_{2,r}\}$, by (\ref{z5}), we have
\begin{eqnarray*}
\sum_{\sigma\in\Lambda_{k,r}}(h(\sigma))^{\frac{\xi_r}{\xi_r+r}}\geq\left\{\begin{array}{ll}
p_0^{\frac{\xi_r}{\xi_r+r}}a(\xi_r)^{l_{1k}}=p_0^{\frac{\xi_r}{\xi_r+r}}\;\;\;\;{\rm if}\;\; \xi_r=\xi_{1,r}\\
b(\xi_r)^{l_{1k}}=1,\;\;\;\;\;\;\;\;\;\;\;\;\;\;\;\;\;\;{\rm if}\;\;\xi_r=\xi_{2,r}\end{array}\right..
\end{eqnarray*}
Thus, by the definition of $\Lambda_{k,r}$ (see (\ref{gamman})), we have
\begin{equation*}
 N_{k,r}\geq p_0^{\frac{\xi_r}{\xi_r+r}}(k\underline{\eta}_r^{-1})^{\frac{\xi_r}{\xi_r+r}}.
\end{equation*}
By Lemma \ref{errorestimate} and H\"{o}lder's inequality with exponent less than one,
\begin{eqnarray*}
 N_{k,r}^{\frac{r}{\xi_r}}e^r_{ N_{k,r},r}(\mu)&\geq& N_{k,r}^{\frac{r}{\xi_r}}\sum_{\sigma\in\Lambda_{k,r}}h(\sigma)
\geq N_{k,r}^{\frac{r}{\xi_r}}\bigg(\sum_{\sigma\in\Lambda_{k,r}}h(\sigma)^{\frac{\xi_r}{\xi_r+r}}\bigg)^{\frac{\xi_r+r}{\xi_r}}
 N_{k,r}^{-\frac{r}{\xi_r}}\\&\geq& (p_0^{\frac{r}{\xi_r+r}})^{\frac{\xi_r+r}{\xi_r}}=p_0.
\end{eqnarray*}
This implies that $\underline{P}_r^{\xi_r}(\mu)>0$. Hence, by Lemma \ref{g13} \cite[Proposition 11.3]{GL:00} (see also \cite{PK:01}), we conclude that
\begin{equation}\label{g15}
\underline{Q}_r^{\xi_r}(\mu)>0,\;\;{\rm and}\;\;\underline{D}_r(\mu)\geq \xi_r.
\end{equation}

Let $s>\xi_r$ be given. Then, $c(s):=\max\{a(s),b(s)\}<1$. By Lemma \ref{s3},
\begin{eqnarray*}
\sum_{\sigma\in\Lambda_{k,r}}(h(\sigma))^{\frac{s}{s+r}}&\leq&\sum_{h=0}^{l_{1k}}a(s)^{l_{1k}-h}b(s)^{h}+\sum_{h=l_{1k}}^{l_{2k}}b(s)^h+b(s)^{l_{1k}}\\
&\leq& l_{1k}c(s)^{l_{1k}}+b(s)^{l_{1k}}(1-b(s))^{-1}+b(s)^{l_{1k}}=:d_k(s).
\end{eqnarray*}
Note that $d_k(s)\to 0$ as $k$ tends to infinity. From this and (\ref{gamman}), we deduce that, there exists some positive real number $d(s)$ such that
\begin{eqnarray*}
 N_{k,r}\leq d(s)(k\underline{\eta}_r^{-2})^{\frac{s}{s+r}}\;\;{\rm for\; all\;large}\;\; k.
\end{eqnarray*}
 Hence, by Lemma \ref{errorestimate}, we have
\begin{eqnarray}\label{z6}
 N_{k,r}^{\frac{r}{s}}e^r_{ N_{k,r},r}(\mu)&\leq& N_{k,r}^{\frac{r}{s}}\sum_{\sigma\in\Lambda_{k,r}}h(\sigma)
\leq N_{k,r}^{\frac{r}{s}}\sum_{\sigma\in\Lambda_{k,r}}(h(\sigma))^{\frac{s}{s+r}}
(k^{-1}\underline{\eta}_r)^{\frac{r}{s+r}}\nonumber\\&\leq& d(s)^{\frac{r}{s}}(k\underline{\eta}_r^{-2})^{\frac{r}{s+r}} d(s)(k^{-1}\underline{\eta}_r )^{\frac{r}{s+r}}
=d(s)^{\frac{s+r}{s}}\underline{\eta}_r^{-\frac{r}{s+r}}.
\end{eqnarray}
It follows that $\overline{P}_r^s(\mu)<\infty$, so $\overline{Q}_r^s(\mu)<\infty$ by Lemma \ref{g13}. As a consequence of this and \cite[Proposition 11.3]{GL:00}, we obtain that $\overline{D}_r(\mu)\leq s$. By the arbitrariness of $s$, we have, $\overline{D}_r(\mu)\leq \xi_r$. Then, in view of (\ref{g15}), we conclude that $D_r(\mu)$ exists and equals $\xi_r$.

(a) Assume that $\xi_r=\xi_{1,r}>\xi_{2,r}$. Then $a(\xi_r)=1$ and $b(\xi_r)<1$. Hence, by using Lemma \ref{s3}, we deduce
\begin{eqnarray*}
\sum_{\sigma\in\Lambda_{k,r}}(h(\sigma))^{\frac{\xi_r}{\xi_r+r}}&\leq&
\sum_{h=0}^{l_{1k}-1}a(\xi_r)^{l_{1k}-h}b(\xi_r)^{h}+\sum_{h=l_{1k}}^{l_{2k}}b(\xi_r)^h+b(\xi_r)^{l_{1k}}\\
&\leq&\sum_{h=0}^{l_{2k}}b(\xi_r)^h+b(\xi_r)^{l_{1k}}\leq\frac{1}{1-b(\xi_r)}+1<\infty.
\end{eqnarray*}
Thus, as we did for (\ref{z6}), one can see that  $\overline{P}_r^{\xi_r}(\mu)<\infty$. By Lemma \ref{g13}, it follows that $\overline{Q}_r^{\xi_r}(\mu)<\infty$. This and (\ref{g15}) complete the proof of (a).

(b) Assume that $\xi_r=\xi_{1,r}=\xi_{2,r}$. By Lemma \ref{errorestimate} and H\"{o}lder's inequality (with exponent less than one), we deduce
\begin{eqnarray*}
e^r_{ N_{k,r},r}(\mu)\geq D\sum_{\sigma\in\Lambda_{k,r}}h(\sigma\geq\bigg(\sum_{\sigma\in\Lambda_{k,r}}h(\sigma)^{\frac{\xi_r}{\xi_r+r}}\bigg)^{\frac{\xi_r+r}{\xi_r}}\cdot N_{k,r}^{-\frac{r}{\xi_r}}.
\end{eqnarray*}
Using this and Lemma \ref{s4}, we conclude that
\begin{eqnarray*}
\underline{P}_r^{\xi_r}(\mu)=\liminf_{k\to\infty} N_{k,r}^{\frac{1}{\xi_r}}e_{ N_{k,r},r}(\mu)\geq \liminf_{k\to\infty}(p_0l_{1k})^{\frac{1}{r}}=\infty.
\end{eqnarray*}
Thus, by Lemma \ref{g13}, Theorem \ref{mthm1} (b) follows.

\begin{remark}{\rm
(r1) In case $\xi_{1,r}<\xi_{2,r}$, it remains open whether the upper quantization coefficient is finite or not.
(r2) By \cite{Zhu:08}, $\mu$ is a self-similar measure associated with $(f_i)_{i=1}^N$ if and only if $p_i+p_0t_i=t_i$ for all $i=1,\ldots, N$. In this case, by \cite{GL:01}, we have,
\begin{equation}\label{g7}
0<\underline{Q}_r^{\xi_r}(\mu)\leq\overline{Q}_r^{\xi_r}(\mu)<\infty.
\end{equation}
Now we show that, (\ref{g7}) can also be deduced from Theorem \ref{mthm1} (a). In fact, we have, $p_i=(1-p_0)t_i$ for all $1\leq i\leq N$. Thus,
\begin{eqnarray*}
\sum_{i=1}^N(p_is_i^r)^{\frac{\xi_{1,r}}{\xi_{1,r}+r}}=(1-p_0)^{\frac{\xi_{1,r}}{\xi_{1,r}+r}}\sum_{i=1}^N(t_is_i^r)^{\frac{\xi_{1,r}}{\xi_{1,r}+r}}
=(1-p_0)^{\frac{\xi_{1,r}}{\xi_{1,r}+r}}<1.
\end{eqnarray*}
This implies that $\xi_{1,r}>\xi_{2,r}$ and (\ref{g7}) follows from Theorem \ref{mthm1} (a).
}\end{remark}


\begin{thebibliography}{16}
\bibitem{Bar:88} Barnsley M.F.:
 Fractals everywhere. Academic Press, New York, London, 1988

\bibitem{BW:82}Bucklew J.A. and Wise G.L.: Multidimensional
asymptotic quantization with $r$th power distortion measures. IEEE
Trans. Inform. Theory \textbf{28}, 239-47 (1982)

\bibitem{Gr:95}Graf S.: On Bandt's tangential distribution for self-similar measures.
Monatsh. Math. \textbf{120}, 223-246 (1995)

\bibitem{GL:00} Graf S. and Luschgy  H.:
Foundations of quantization for probability distributons. Lecture
Notes in Mathematics \textbf{1730}, Springer-Verlag, 2000.

\bibitem{GL:01} Graf S. and Luschgy  H.:
The assmptotics of the  quantization errors for self-similar
probabilities.  Real Anal. Exchange. \textbf{26}, 795-810 (2000/2001)

\bibitem{GL:04}Graf S. and Luschgy  H.:
Quantization for probabilitiy measures with respect to the geometric
mean error. Math. Proc. Camb. Phil. Soc. \textbf{136}, 687-717 (2004)

\bibitem{GL:12}Graf S., Luschgy  H. and  Pag\'{e}s G.: The local quantization behavior of absolutely
continuous probabilities. Ann. Probab. \textbf{40}, 1795-1828 (2012)

\bibitem{GN:98}Gray R. and Neuhoff D.:
Quantization IEEE Trans. Inform. Theory \textbf{44}, 2325-2383 (1998)

\bibitem{Gru:04}Gruber, Peter M.: Optimum quantization and its applications. Adv Math. \textbf{186}, 456-497 (2004)

\bibitem{Hu:81} Hutchinson J.E.: Fractals and self-similarity. Indiana Univ. Math. J.
\textbf{30}, 713-747 (1981).

\bibitem{HL:96} Hua S. and Li, W.X.:
Packing dimension of generalized Moran sets, Progr.
Natur. Sci. \textbf{6}, 148-152 (1996)

\bibitem{Kr:08}Kreitmeier W.: Optimal quantization for dyadic
homogeneous Cantor distributions. Math. Nachr. \textbf{281},
1307-1327 (2008)

\bibitem{Las:06} Lasota A.:
A variational principle for fractal dimensions. Nonlinear Analysis
\textbf{64}, 618-628 (2006)

\bibitem{Olsen:08}Olsen L.  and Snigireva  N.: Multifractal spectra of in-homogenous self-similar measures.
Indiana Univ. Math. J. \textbf{57},1887-1841 (2008)

\bibitem{PG:98}Pag\`{e}s G.: A space quantization method for numerical integration. J. Comput. Appl. Math. \textbf{89}, 1-38  (1998)

\bibitem{PK:01}P\"{o}tzelberger K. : The quantization dimension of distributions,
Math. Proc. Camb. Phil. Soc. \textbf{131}, 507-519 (2001)

\bibitem{RMK:13}Roychowdhury M.:  Quantization dimension estimate of inhomogeneous self-similar measures.
Bull. Pol. Acad. Sci. Math. \textbf{61}, 35-45 (2013)

\bibitem{Schief:94}A. Schief, Separation properties for self-similar sets. Proc. Amer. Math. Soc. \textbf{122}, 111-115 (1994)

\bibitem{Za:82} Zador P.L.:
Asymptotic quantization error of continuous signals and the
quantization dimension. IEEE Trans. Inform. Theory \textbf{28},
139-149 (1982)

\bibitem{Zhu:08}Zhu S.: Quantization dimension for condensation systems, Math. Z. \textbf{259} , 33-43 (2008)

\bibitem{Zhu:12}Zhu S.: On the upper and lower quantization coefficient for probability measures on multiscale Moran sets.
Chaos, Solitons \& Fractals \textbf{45}, 1437-1443  (2012)



\end{thebibliography}
\end{document}